\documentclass[twoside]{irmaems}
\usepackage{amssymb} 
\usepackage{amsmath} 
\usepackage{latexsym}
\usepackage[all]{xy}

\renewcommand\theequation{\arabic{section}.\arabic{equation}}

\theoremstyle{definition} 

 \newtheorem{definition}{Definition}[section]
 \newtheorem{remark}[definition]{Remark}

\theoremstyle{plain}      

 \newtheorem{proposition}[definition]{Proposition}
 \newtheorem{theorem}[definition]{Theorem}
 \newtheorem{corollary}[definition]{Corollary}
 \newtheorem{lemma}[definition]{Lemma}

\begin{document}

\title{Meyer functions and the signatures of fibered 4-manifolds}

\author{Yusuke Kuno\thanks{
Work partially supported by JSPS Research Fellowships for Young Scientists
(22$\cdot$4810)}}

\address{
Department of Mathematics, Tsuda College,\\
2-1-1 Tsuda-Machi, Kodaira-shi,
Tokyo 187-8577 JAPAN \\
email:\,\tt{kunotti@tsuda.ac.jp}
}

\maketitle

\begin{abstract}
We give a survey on Meyer functions, with emphasis on their
application to the signatures of fibered 4-manifolds.
\end{abstract}

\begin{classification}
14D05, 20F34, 32G15; 57N13.
\end{classification}

\begin{keywords}
the signature cocycle, Meyer function, local signature.
\end{keywords}

\tableofcontents   

\section{Introduction}

In this chapter, we give a survey on secondary invariants called {\it Meyer functions}
with emphasis on their application to the signatures of fibered 4-manifolds.
These secondary invariants are associated to the vanishing of
the primary invariant called the {\it first MMM class} $e_1$, the first
in a series of characteristic classes of surface bundles
\cite{Miller} \cite{Morita1} \cite{Mumford}.
There have been known various representatives of $e_1$ coming from different geometric
contexts, as group 2-cocycles on the mapping class group or differential 2-forms
on the moduli space of curves (see \cite{Kaw2}, especially for the latter).
The view point we take here is the signature of surface bundles over surfaces, and
we work with the {\it signature cocycle} $\tau_g$ introduced by W. Meyer \cite{Mey2}
(and by Turaev \cite{Turaev} independently)
a $\mathbb{Z}$-valued 2-cocycle of the mapping class group $\mathcal{M}_g$ of
a closed oriented surface of genus $g$, whose cohomology class is proportional to $e_1$.

As was shown by Meyer, if $g=1$ or $2$, the cocycle $\tau_g$ is the coboundary
of a unique $\mathbb{Q}$-valued 1-cochain $\phi_g$ of $\mathcal{M}_g$. The existence
of such a 1-cochain implies that over the rationals, $e_1$ of a surface bundle with fiber a surface of
genus $1$ or $2$ vanishes. The uniqueness of $\phi_g$ follows from the fact
$H^1(\mathcal{M}_g;\mathbb{Q})=0$. These 1-cochains are called the Meyer functions
of genus $1$ or $2$.
Meyer \cite{Mey2} extensively studied the case of genus $1$ and gave an explicit
formula for $\phi_1$ which involves the Dedekind sums.
In \cite{Ati2}, Atiyah reproved Meyer's formula by a quite different method and
also showed various number theoretic or differential geometric aspects of $\phi_1$.

In \S 2, we recall basic results of Meyer and Atiyah with a sketch of proof
for several assertions. In \S 3, we mention an application of Meyer functions
to localization of the signature of fibered 4-manifolds.
This topic has been studied also from algebro-geometric point of view, which
we shall mention in \S \ref{subsec:Horikawa}.
Recently, various higher genera or higher dimensional analogues of
$\phi_1$ have been considered and a part of Atiyah's result has been generalized
to these generalizations. In \S 4, we present three examples of these generalizations.

Some conventions about surface bundles follow.
Throughout this chapter $g$ is an integer $\ge 1$.
Let $\Sigma_g$ be a closed oriented $C^{\infty}$-surface of genus $g$.
By a {\it $\Sigma_g$-bundle} we mean a smooth fiber bundle $\pi\colon E\to B$ over a
$C^{\infty}$-manifold $B$ with fiber $\Sigma_g$ such that the fibers are coherently oriented:
the tangent bundle along the fibers $T\pi:=\{ v\in TE; \pi_*(v)=0\}$ is oriented.
The transition functions of such bundles take values in ${\rm Diff}^+(\Sigma_g)$,
the group of orientation preserving diffeomorphisms of $\Sigma_g$ endowed with $C^{\infty}$-topology.
The {\it mapping class group} $\mathcal{M}_g:=\pi_0({\rm Diff}^+(\Sigma_g))$ is the group
of connected components of ${\rm Diff}^+(\Sigma_g)$. In other words, $\mathcal{M}_g$
is the quotient group ${\rm Diff}^+(\Sigma_g)/{\rm Diff}_0(\Sigma_g)$, where
${\rm Diff}_0(\Sigma_g)$ is the group of diffeomorphisms isotopic to the identity.

For a $\Sigma_g$-bundle $\pi\colon E\to B$ over a path connectd space $B$,
the associated is (the conjugacy class of) a homomorphism $\chi\colon \pi_1(B)\to \mathcal{M}_g$
called the {\it monodromy}. This correspondence is defined as the composite
\begin{align}
\{ \Sigma_g {\rm -bundles\ over\ } B \}/{\rm isom}
&= [B,B{\rm Diff}^+(\Sigma_g)] \nonumber \\
&\to {\rm Hom}(\pi_1(B),\mathcal{M}_g)/{\rm conj}. \label{eq:mon}
\end{align}
Namely, if $f\colon B \to B{\rm Diff}^+(\Sigma_g)$ is a classifying map
of $\pi\colon E\to B$, then $\chi=f_*$, the induced map from $\pi_1(B)$ to
$\pi_1(B{\rm Diff}^+(\Sigma_g))=\pi_0({\rm Diff}^+(\Sigma_g))=\mathcal{M}_g$.
To be more careful about the base points and to give a more direct description,
choose a base point $b_0\in B$ and fix an orientation preserving diffeomorphism
$\varphi\colon \Sigma_g \to \pi^{-1}(b_0)$. Let $\ell \colon [0,1]\to B$ be a based loop.
Since $[0,1]$ is contractible, the pull back $\ell^*(E)\to [0,1]$ of $\pi\colon E\to B$ is
a trivial $\Sigma_g$-bundle. Hence there exist a trivialization
$\Phi\colon \Sigma_g \times [0,1]\to \ell^*(E)$ such that $\Phi(x,0)=\varphi(x)$.
In this setting, $\chi\colon \pi_1(B,b_0)\to \mathcal{M}_g$ is given by
$\chi([\ell])=[\Phi(x,1)^{-1}\circ \varphi]$. Here our convention is:
1) for any two mapping classes $f_1$ and $f_2$, their multiplication $f_1\circ f_2$ means
that $f_2$ is applied first, 2) for any two homotopy classes of based loops
$\ell_1$ and $\ell_2$, their product $\ell_1\cdot \ell_2$ means that $\ell_1$ is traversed first.

By the result of Earle-Eells \cite{EE}, if $g\ge 2$ the space ${\rm Diff}_0(\Sigma_g)$
is contractible, so the classifying space $B{\rm Diff}^+(\Sigma_g)$
is a $K(\mathcal{M}_g,1)$-space. Hence the map (\ref{eq:mon}) is a bijection.
If $g=1$, then $\Sigma_1=T^2$, the two torus. The embedding
$T^2 \hookrightarrow {\rm Diff}_0(T^2)$ as parallel translations is a homotopy equivalence,
and $\mathcal{M}_1$ is isomorphic to $SL(2;\mathbb{Z})$. Thus we have a fibration
$B{\rm Diff}^+(T^2)\to BSL(2;\mathbb{Z})=K(SL(2;\mathbb{Z}),1)$ with fiber
$BT^2=\mathbb{C}P^{\infty}\times \mathbb{C}P^{\infty}$. In particular,
by elementary obstruction theory, it follows that if the base space $B$ has
a homotopy type of a 1-dimensional CW complex, then the isomorphism class
of $T^2$-bundles over $B$ is also classified well by monodromies:
(\ref{eq:mon}) is bijective.

\section{The signature cocycle and Meyer's theorem}

In this section we review the signature cocycle, its variants, and
the original version of Meyer functions, i.e., the Meyer function of genus 1 and 2.

\subsection{Prehistory}

In study of the topology of fiber bundles, a basic question is how
the topological invariants of the total space, the base space and the fiber are related.
In 50's Chern, Hirzebruch and Serre studied the signature of the total space
of a fiber bundle, by an application of the Serre spectral sequence.
Recall that the {\it signature} of a compact oriented manifold
$M$ of dimension $4n$ (possibly with boundary), denoted by ${\rm Sign}(M)$, is the signature of the
intersection form $H_{2n}(M;\mathbb{R}) \times H_{2n}(M;\mathbb{R}) \to \mathbb{R}$,
which is a symmetric bilinear form. If the dimension of $M$ is not a multiple
of 4, we understand that the signature of $M$ is zero.

\begin{theorem}[Chern-Hirzebruch-Serre \cite{CHS}]
\label{thm:CHS}
Let $E$ and $B$ be closed oriented manifolds and
$E\to B$ a fiber bundle with fiber a closed oriented manifold $F$.
We arrange that the orientation of $F$ is compatible with those of $E$ and $B$.
If $\pi_1(B)$ trivially acts on the homology $H_*(F;\mathbb{R})$,
then the signature of $E$ is the product of the signatures of $B$ and $F$:
${\rm Sign}(E)={\rm Sign}(B){\rm Sign}(F)$.
\end{theorem}

The assumption that $\pi_1(B)$ trivially acts on the homology of the fiber
is crucial, and the conclusion of the theorem does not hold in general.
Indeed, Atiyah \cite{Ati1} and Kodaira \cite{Kod}
independently constructed an algebraic surface with non-zero signature,
which is the total space of a complex analytic family of
compact Riemann surfaces over a compact Riemann surface.
Their method uses branched covering of algebraic surfaces, and can be
used to produce examples such that the genus of the fiber can be taken
arbitrarily integers $\ge 4$.

One important consequence is that there are non-trivial characteristic
classes of surface bundles. In fact,  since the signature of a manifold which is the boundary
of some manifold is zero, the map
$${\rm Sign}\colon \Omega_2(B{\rm Diff}^+(\Sigma_g))\to \mathbb{Z},
\quad [f]\mapsto {\rm Sign}(f^*\xi)$$
is well-defined. Here $\Omega_2(X)$ is the second oriented bordism group of a space $X$
(hence its element is represented by some continuous map $f$ from a closed oriented surface
to $X$) and $\xi$ is a universal $\Sigma_g$-bundle over the classifying space $B{\rm Diff}^+(\Sigma_g)$.
Since $\Omega_2(X)$ is naturally isomorphic to $H_2(X;\mathbb{Z})$, the map ${\rm Sign}$ becomes
an element in ${\rm Hom}(H_2(B{\rm Diff}^+(\Sigma_g)), \mathbb{Z})$, and the examples by
Atiyah and Kodaira shows that the map ${\rm Sign}$ is non-trivial. Hence
$H^2(B{\rm Diff}^+(\Sigma_g); \mathbb{Z})\cong H^2(\mathcal{M}_g; \mathbb{Z})$
is non-trivial and contains an element of infinite order, provided $g\ge 4$.
As we recall in the following, Meyer showed that this non-triviality holds when $g\ge3$.

\subsection{The signature cocycle}

W. Meyer \cite{Mey1} \cite{Mey2} studied the signature of surface bundles over surfaces
and introduced the signature cocycle. The basic idea of Meyer is
to decompose the base space into simple pieces: pairs of pants. 

Let $\Sigma_{0,n}$ be a compact surface obtained from the two sphere
by removing $n$ open disks with embedded disjoint closures. Specifying an orientation
of $\Sigma_{0,n}$ and a base point $*\in {\rm Int}(\Sigma_{0,n})$, we take $n$ based loops
$\ell_1,\ldots, \ell_n\in \pi_1(\Sigma_{0,n},*)$ such that each $\ell_i$ is freely homotopic to
one of the boundaries with the counter-clockwise orientation, and the relation $\ell_1\cdots \ell_n=1
\in \pi_1(\Sigma_{0,n},*)$ holds. The group $\pi_1(\Sigma_{0,n},*)$
is free of rank $n-1$, generated by any $n-1$ of $\ell_1,\ldots,\ell_n$.
The surface $P=\Sigma_{0,3}$ is called a {\it pair of pants}.

Given $f_1,\ldots,f_{n-1}\in \mathcal{M}_g$, consider a $\Sigma_g$-bundle
$\pi \colon E(f_1,\ldots,f_{n-1})\to \Sigma_{0,n}$ with $\pi^{-1}(*)=\Sigma_g$ whose monodromy
$\chi\colon \pi_1(\Sigma_{0,n},*)\to \mathcal{M}_g$ sends $\ell_i$ to $f_i$ ($i=1,\ldots,n-1$).
Since $\Sigma_{0,n}$ is homotopy equivalent to a 1-dimensional CW complex,
such a bundle exists and is unique up to isomorphism (see \S 1).
The total space $E(f_1,\ldots,f_{n-1})$ is a compact oriented 4-manifold with boundary.

\begin{definition}
\label{dfn:sig-cocycle}
The {\it signature cocycle} $\tau_g\colon \mathcal{M}_g \times \mathcal{M}_g \to \mathbb{Z}$ is defined by
$$\tau_g(f_1,f_2):={\rm Sign}(E(f_1,f_2)), \quad f_1,f_2\in \mathcal{M}_g.$$
\end{definition}

The map $\tau_g$ is actually a normalized two cocycle of $\mathcal{M}_g$.

\begin{lemma}
\label{lem:tau}
For $f_1,f_2,f_3\in \mathcal{M}_g$, we have
\begin{enumerate}
\item $\tau_g(f_1f_2,f_3)+\tau_g(f_1,f_2)=\tau_g(f_1,f_2f_3)+\tau_g(f_2,f_3)$;
\item $\tau_g(f_1,1)=\tau_g(1,f_1)=\tau_g(f_1,f_1^{-1})=0$;
\item $\tau_g(f_1^{-1},f_2^{-1})=-\tau_g(f_1,f_2)$;
\item $\tau_g(f_1,f_2)=\tau_g(f_2,f_1)$;
\item $\tau_g(f_3f_1f_3^{-1},f_3f_2f_3^{-1})=\tau_g(f_1,f_2)$.
\end{enumerate}
\end{lemma}

\begin{proof}[sketch of proof] Recall the {\it Novikov additivity} of the signature.
Let $M_1$ and $M_2$ be compact oriented manifolds of the same dimension,
$Y_1$ and $Y_2$ closed and open submanifolds of $\partial M_1$ and $\partial M_2$,
respectively, and $\varphi\colon Y_1\to Y_2$ an orientation reversing homeomorphism.
Then the signature of the glued manifold $M_1\cup_{\varphi} M_2$ is the sum of
the signatures of $M_1$ and $M_2$.

We only give the proof of (1), the cocycle condition for $\tau_g$.
Consider a $\Sigma_g$-bundle $\pi\colon E(f_1,f_2,f_3)\to \Sigma_{0,4}$
and let $C_1,C_2\subset \Sigma_{0,4}$ be essential simple closed curves
intersecting each other in two points,
such that $C_1$ cuts $\Sigma_{0,4}$ into two pairs of pants and the boundary of one of the two
contains the free homotopy class of $\ell_1$ and $\ell_2$. According to the decomposition of the base space,
the total space $E(f_1,f_2,f_3)$ can be written as a connected sum of $E(f_1f_2,f_3)$ and $E(f_1,f_2)$.
By the Novikov additivity of the signature, we obtain ${\rm Sign}(E(f_1,f_2,f_3))=\tau_g(f_1f_2,f_3)+\tau_g(f_1,f_2)$.
On the other hand cutting along $C_2$ and arguing similarly, we obtain ${\rm Sign}(E(f_1,f_2,f_3))=\tau_g(f_1,f_2f_3)+\tau_g(f_2,f_3)$.
\end{proof}

The signature cocycle has a purely algebraic description.
We denote by $I_n$ the $n\times n$ identity matrix.
The {\it integral symplectic group} $Sp(2g;\mathbb{Z})$, also called the {\it Siegel modular group}, is defined by
$$Sp(2g;\mathbb{Z}):=\{ A \in GL(2g;\mathbb{Z}); \ ^t AJA=J \},$$
where $J=\left( \begin{array}{cc} 0 & I_g \\ -I_g & 0 \end{array}\right)$ and $I_g$ is the $g\times g$ identity matrix.
Fix a symplectic basis of $H_1(\Sigma_g; \mathbb{Z})$, i.e., elements $A_1,\ldots,A_g,B_1,\ldots,B_g\in H_1(\Sigma_g;\mathbb{Z})$
whose algebraic itersection numbers satisfy
$$(A_i\cdot B_j)=\delta_{ij},\quad (A_i\cdot A_j)=(B_i\cdot B_j)=0.$$
In terms of a symplectic basis, the (left) action of $\mathcal{M}_g$ on $H_1(\Sigma_g;\mathbb{Z})$ is expressed
as matrices and we get a (surjective) group homomorphism
\begin{equation}
\label{eq:rho}
\rho\colon \mathcal{M}_g \to Sp(2g;\mathbb{Z}).
\end{equation}
Given $A,B \in Sp(2g;\mathbb{Z})$, consider a $\mathbb{R}$-linear space
$$V_{A,B}:=\{ (x,y)\in \mathbb{R}^{2g} \oplus \mathbb{R}^{2g}; (A^{-1}-I_{2g})x+(B-I_{2g})y=0 \}$$
and a bilinear form $\langle \ ,\ \rangle_{A,B} \colon V_{A,B}\times V_{A,B}\to \mathbb{R}$
defined by
$$\langle (x,y),(x^{\prime},y^{\prime}) \rangle_{A,B}:=\ ^t (x+y)J(I_{2g}-B)y^{\prime}.$$
It turns out that $\langle \ ,\ \rangle_{A,B}$ is symmetric hence its signature
${\rm Sign}(V_{A,B},\langle \ ,\ \rangle_{A,B})$ is defined. We denote by
$\tau_g^{\rm sp}$ the map $Sp(2g;\mathbb{Z})\times Sp(2g;\mathbb{Z}) \to \mathbb{Z}, (A,B)\mapsto
{\rm Sign}(V_{A,B},\langle \ ,\ \rangle_{A,B})$. Note that $\tau_g^{\rm sp}$ is naturally
defined on the Lie group $Sp(2g;\mathbb{R})$.

\begin{theorem}[Meyer \cite{Mey1}]
The signature cocycle on $\mathcal{M}_g$ is the pull-back of $\tau_g^{\rm sp}$
on $Sp(2g;\mathbb{Z})$, i.e., for any $f_1,f_2\in \mathcal{M}_g$, we have
$$\tau_g(f_1,f_2):={\rm Sign}(V_{\rho(f_1),\rho(f_2)},\langle \ ,\ \rangle_{\rho(f_1),\rho(f_2)}).$$
\end{theorem}

\begin{proof}[sketch of proof] The proof proceeds following the proof of Theorem \ref{thm:CHS}.
Consider the Serre cohomology spectral sequence of $E(f_1,f_2)\to P$.
The $E_2$ page is $E_2^{p,q}=H^p(P,\partial P;\mathcal{H}^q(\Sigma_g;\mathbb{R}))$,
where $\mathcal{H}^q(\Sigma_g;\mathbb{R})$ denotes the local system on $P$ whose
stalk at $b\in P$ is the cohomology of $\pi^{-1}(b)$. On the other hand each page $E_r$ is
a {\it Poincar\'e ring} in the sense of \cite{CHS}, in particular its
signature ${\rm Sign}(E_r)$ is defined. The proof is done through three steps:
(1) to show that ${\rm Sign}(E_r)={\rm Sign}(E_{r+1})$,
(2) to show that ${\rm Sign}(E_{\infty})={\rm Sign}(E(f_1,f_2))$, and
(3) to show that ${\rm Sign}(E_2)={\rm Sign}(V_{\rho(f_1),\rho(f_2)}
\langle \ ,\ \rangle_{\rho(f_1),\rho(f_2)})$. To prove the last step,
by taking a simplicial decomposition of $P$, Meyer \cite{Mey1} observed that
$E_2^{1,1}=H^1(P,\partial P;\mathcal{H}^1(\Sigma_g;\mathbb{R}))$ is isomorphic to
$V_{\rho(f_1),\rho(f_2)}$, and the cup product on the former corresponds
to $\langle \ ,\ \rangle_{\rho(f_1),\rho(f_2)}$.
\end{proof}

The signature cocycle is independently introduced by Turaev \cite{Turaev}.
He gave another algebraic description for $\tau_g^{\rm sp}$ and directly proved
that $\tau_g^{\rm sp}$ is a normalized two cocycle. He also discusses a relation
with the Maslov index. For coincidence of the definition of $\tau_g^{\rm sp}$
by Meyer and Turaev, see Endo-Nagami \cite{EN} Appendix.

\begin{remark}
\label{var-sig}
Let $M$ be a closed oriented manifold of dimension $4n-2$ and
$\pi\colon E\to B$ an oriented $M$-bundle with $B$ path connected.
By mimicking Definition \ref{dfn:sig-cocycle}, i.e., by constructing a $M$-bundle
over $P$ and taking the signature of the total space, we obtain a normalized 2-cocycle
$c_M\colon \pi_1(B)\times \pi_1(B)\to \mathbb{Z}$.
In another direction, Atiyah \cite{Ati2} introduced the signature cocycle
on the Lie group $U(p,q)$, the unitary group of the Hermitian form with signature $(p,q)$.
The restriction to $Sp(2g;\mathbb{R}) \subset U(p,p)$ is $\tau_g^{\rm sp}$.
\end{remark}

\subsection{Evaluation of the signature class}
The cocycle $\tau_g\in Z^2(\mathcal{M}_g;\mathbb{Z})$ determines
a cohomology class $[\tau_g]\in H^2(\mathcal{M}_g;\mathbb{Z})$, which here we call
the signature class. We give a combinatorial method to compute the order of $[\tau_g]$.
Following Meyer \cite{Mey2}, we consider the following slightly general situation:
let $G$ be a group and $k\colon G\times G \to \mathbb{Z}$ a normalized 2-cocycle
satisfying $z(x,x^{-1})=0$ for any $x\in G$. Suppose a presentation of $G$ is given.
Namely $G$ fits into an exact sequence
$$1\to R \to F \overset{\varpi}{\to} G \to 1$$
where $F$ is the free group generated by a set $\{ e_i \}_{i\in I}$. Any $x\in F$
can be written as $x=x_1x_2\cdots x_m$, where $x_j\in \{ e_i\} \cup \{ e_i^{-1} \}$.
Define $c\colon F\to \mathbb{Z}$ by
$$c(x):=\sum_{j=1}^m z(\varpi(x_1\cdots x_{j-1}), \varpi(x_j)).$$
It follows that $c$ is well-defined and $\delta c=-\varpi^*z$, i.e.,
$c(xy)=c(x)+c(y)+z(\varpi(x),\varpi(y))$ for $x,y\in F$. Moreover, $c$ is a class
function: $c(yxy^{-1})=c(x)$ for $x,y\in F$.
The 1-cochain $c$ is involved in a commutative diagram
$$\xymatrix{
H_2(G; \mathbb{Z}) \ar[r]^{ev([z])} & \mathbb{Z} \\
R\cap [F,F]/[R,F] \ar[u]^{\cong} \ar[ur]_c }$$
where the vertical isomorphisms is due to Hopf's formula (see \cite{Bro}) and the upper right arrow
is the evaluation map $ev([z])\colon H_2(G;\mathbb{Z}) \to \mathbb{Z}$ by $[z]$.
For $i\in I$, let $e_i^* \colon F\to \mathbb{Z}$ be the map counting
the total exponents of $e_i$ in elements of $F$.

\begin{proposition}[Meyer \cite{Mey2}]
\label{prop:comb.cri}
For $m\in \mathbb{Z}\setminus \{ 0\}$, the order of $[z]\in H^2(G;\mathbb{Z})$ divides $m$
if and only if there exists $\{ m_i\}_{i\in I}\subset \mathbb{Z}$ such that
$mc|_R=\sum_{i\in I}m_ie_i^*|_R$. In particular, if $R$ is the normal closure of a set $\{ r_j \}_{j\in J}\subset F$,
then $[z]=0\in H^2(G;\mathbb{\mathbb{Q}})$ if and only if the liner equation
$c(r_j)=\sum_{i\in I}m_ie_i^*(r_j), j\in J$, has a solution $\{ m_i \}_{i\in I}\subset \mathbb{Q}$.
\end{proposition}

The proof is straightforward, but we briefly mention ``if" part. Take $\{ m_i\}_{i\in I}$
satisfying the condition. Consider the $(1/n)\mathbb{Z}$-valued 1-cochain $c_1:=c-(1/n)\sum_{i\in I}n_i e_i^*$ of $F$.
Then it turns out that $c_1$ descends to a 1-cochain $\overline{c}_1\colon G=F/R\to (1/n)\mathbb{Z}$.
In fact, for $x\in F$ and $r\in R$, we have
\begin{align*}
c_1(xr) &=c(x)+c(r)+\varpi^*z(x,r)-\frac{1}{n}\sum_{i\in I}n_i(e_i^*(x)+e_i^*(r)) \\
&=c(x)-\frac{1}{n}\sum_{i\in I}n_ie_i^*(x)=c_1(x)
\end{align*}
(we use $\varpi(r)=1$). Since $\varpi$ is surjective, it follows that $\delta \overline{c}_1=-z$.

In a special situation, this criterion becomes simpler.
Let $Art(\mathcal{G})$ be a (small) {\it Artin group} associated to
a connected graph $\mathcal{G}$ without loops. This means that $Art(\mathcal{G})$ is generated by
the vertex set $\{ a_i \}_{i\in I}$ of $\mathcal{G}$, subject to the defining relations
$a_ia_ja_i=a_ja_ia_j$ if $a_i$ and $a_j$ are adjacent, and $a_ia_j=a_ja_i$ if not.
Further let $\{ r_j \}_{j\in J}$ be a set of words in $\{ a_i\}_i$. We shall consider the
case $G$ is the group obtained by adding relations $r_j=1$, $j\in J$ to $Art(\mathcal{G})$.
Suppose there exists $\{ m_i\}_{i\in I}\subset \mathbb{Q}$ satisfying the condition of
Proposition \ref{prop:comb.cri}, and let $a_k$ and $a_{\ell}$ be adjacent vertices of $\mathcal{G}$.
Now we have $r_{k,\ell}:=a_ka_{\ell}a_ka_{\ell}^{-1}a_k^{-1}a_{\ell}^{-1}\in R$, and
\begin{align*}
c(r_{k,\ell}) &=
c(a_k)+c((a_{\ell}a_k)a_{\ell}^{-1}(a_{\ell}a_k)^{-1})+z(\varpi(a_k),\varpi(a_k)^{-1}) \\
&=c(a_k)+c(a_{\ell}^{-1})=0.
\end{align*}
Here we use the condition $z(x,x^{-1})=0$ and the fact that $c$ is a class function.
On the other hand, we have $\sum_{i\in I}m_ie_i^*(r_{k,\ell})=m_k-m_{\ell}$. Therefore
we obtain $m_k=m_{\ell}$. Since $\mathcal{G}$ is connected, we conclude
$m_k=m_{\ell}$ for any $k,\ell \in I$. In summary, we have the following.

\begin{proposition}
\label{prop:Art}
Suppose $G$ is the quotient of an Artin group as above, and let $z\in Z^2(G;\mathbb{Z})$
be a normalized 2-cocycle with $z(x,x^{-1})=0$ for any $x\in G$.
\begin{enumerate}
\item For $n\in \mathbb{N}$, $n[z]=0\in H^2(G;\mathbb{Z})$ if and
only if there exist $m\in \mathbb{Z}$ such that
$n\cdot c(r_j)=m\cdot \alpha(r_j)$ for all $j\in J$.
\item In the situation of (1), the 1-cochain $\phi\colon G\to (1/n)\mathbb{Z}$
defined by $\phi(\varpi(x))=-c(x)+(m/n)\alpha(x)$, $x\in F$ is well-defined. Moreover,
$\delta \phi=z$.
\end{enumerate}
Here $\alpha\colon F\to \mathbb{Z}$ is a homomorhism given by $\alpha(a_i)=1$ for $i\in I$.
\end{proposition}

For example, the mapping class group admits a presentation as the quotient
of an Artin group where the relation $a_ia_ja_i=a_ja_ia_j$ corresponds to the braid
relation among two Dehn twists. Thus we can apply this proposition.

\subsection{Meyer's theorems}

Using the combinatorial criterion in the previous section, Meyer determined
the order of the cohomology class $[\tau_g]\in H^2(\mathcal{M}_g;\mathbb{Z})$.

\begin{theorem}[Meyer \cite{Mey2}, Satz 2]
\label{thm:Satz2}
The order of $[\tau_1]$ is 3, the order of $[\tau_2]$ is 5, and
the order of $[\tau_g]$ is infinite if $g\ge 3$.
\end{theorem}

To settle the case $g=1$ and $2$, Meyer used
a classical presentation of $\mathcal{M}_1\cong SL(2;\mathbb{Z})$
and a presentation of $\mathcal{M}_2$ by Birman-Hilden \cite{BH}.
For $g\ge 3$, no finite presentation of $\mathcal{M}_g$ was known
at that time. Still, using some of the known relations and showing that
$[\tau_g]$ is divisible by 4, Meyer proved that the image of $ev([\tau_g])$ is $4\mathbb{Z}$.
We remark that by the Hirzebruch signature formula, we have
$e_1=3[\tau_g]\in H^2(\mathcal{M}_g;\mathbb{Z})$.

\begin{remark}
Nowadays several finite presentations of $\mathcal{M}_g$ for $g\ge 3$ are known.
Using one of them, say the one due to Wajnryb \cite{Waj}, one can directly show
that the image of $ev([\tau_g])$ is $4\mathbb{Z}$.
\end{remark}

The following is an immediate consequence of Theorem \ref{thm:Satz2}.

\begin{theorem}[Meyer \cite{Mey2}, Satz 3]
\label{thm:Satz3}
\begin{enumerate}
\item If $g\le 2$, the signature of the total space of any $\Sigma_g$-bundle
over a closed oriented surface is zero. 
\item If $g\ge 3$, the signature of the total space of a $\Sigma_g$-bundle over a closed oriented surface
is a multiple of $4$. Conversely, for any $g\ge 3$ and $n\in 4\mathbb{Z}$, there exist a $\Sigma_g$-bundle
$E\to B$ over a closed oriented surface with ${\rm Sign}(E)=n$.
\end{enumerate}
\end{theorem}

As a consequence of Theorem \ref{thm:Satz2}, there exist 1-cochains
$\phi_1\colon \mathcal{M}_1 \to (1/3)\mathbb{Z}$ and
$\phi_2\colon \mathcal{M}_2 \to (1/5)\mathbb{Z}$ such that
$\delta \phi_1=\tau_1$ and $\delta \phi_2=\tau_2$. Here
for a 1-cochain $\phi\colon G\to A$ with coefficient in an abelian group
$A$, its {\it coboundary} $\delta \phi$ is a map from $G\times G$ to $A$ given by
$\delta \phi(x,y)=\phi(x)-\phi(xy)+\phi(y)$
(for terminologies of cohomology of groups, see for example, \cite{Bro}).
Thus the condition $\delta \phi_g=\tau_g$ ($g=1$ or $2$) is equivalent to
\begin{equation}
\label{eq:dphi}
\tau_g(x,y)=\phi_g(x)-\phi_g(xy)+\phi_g(y), \quad x,y\in \mathcal{M}_g.
\end{equation}
Moreover, since $H^1(\mathcal{M}_1;\mathbb{Q})=H^1(\mathcal{M}_2;\mathbb{Q})=0$,
such 1-cochains are unique and characterized by (\ref{eq:dphi}).
The 1-cochain $\phi_1$ (resp. $\phi_2$) is called the {\it Meyer function of
genus 1} (resp. {\it of genus 2}).

The following lemma can be directly proved by Lemma \ref{lem:tau} and (\ref{eq:dphi}).

\begin{lemma}
\label{lem:phi}
The Meyer functions $\phi_1$ and $\phi_2$ satisfy the following properties:
for $x,y\in \mathcal{M}_g$ ($g=1$ or $2$),
\begin{enumerate}
\item $\phi_g(1)=0$;
\item $\phi_g(x^{-1})=-\phi_g(x)$;
\item $\phi_g(yxy^{-1})=\phi_g(x)$.
\end{enumerate}
\end{lemma}

Consider a surface bundle over a compact oriented surface. Then the values of $\phi_g$
around a boundary circle (which is well-defined by Lemma \ref{lem:phi} (3)) is interpreted
as signature defects.

\begin{proposition}
\label{prop:s(E)}
Suppose $g=1$ or $2$ and let $\pi\colon E\to B$ be a $\Sigma_g$-bundle over a compact oriented surface $B$
with boundary components $\partial B_i$, $i\in I$. Then
$${\rm Sign}(E)=\sum_{i\in I} \phi_g(x_i),$$
where $x_i\in \mathcal{M}_g$ is the monodromy along the boundary component $\partial B_i$ with
the counter-clockwise orientation. 
\end{proposition}

\begin{proof}[sketch of proof]
Take a pants decomposition of $B$. By the Novikov additivity of the signature, ${\rm Sign}(E)$ is
the sum of the signatures of the components, which is expressed in terms of $\tau_g$.
Using (\ref{eq:dphi}), we obtain the formula.
\end{proof}

Meyer extensively studied the function $\phi_1$ and gave its explicit formula.
Note that the mapping class group $\mathcal{M}_1$ is isomorphic to $SL(2;\mathbb{Z})
=Sp(2;\mathbb{Z})$ by the homomorphism (\ref{eq:rho}).
To state his result, let us prepare some notations.
The {\it Rademacher function} \cite{Rad} is a map
$\Psi\colon SL(2;\mathbb{Z}) \to \mathbb{Q}$ defined by
$$\Psi \left( \left( \begin{array}{cc} a & b \\ c & d \end{array} \right) \right)=
\left\{
\begin{aligned}
& \displaystyle\frac{a+d}{c}-12 {\rm sign}(c)s(a,c)-3{\rm sign}(c(a+d)) & {\rm if\ } c\neq 0, \\
& \displaystyle\frac{b}{d} & {\rm if\ } c=0.
\end{aligned}
\right.
$$
Here ${\rm sign}(x)\in \{0, \pm 1\}$ is the sign of $x$ if $x\neq 0$,
$0$ if $x=0$, and $s(a,c)$ is the {\it Dedekind sum}
$$s(a,c):=\sum_{k {\rm mod\ } |c|} \left( \left( \frac{ak}{c} \right) \right)
\left( \left( \frac{k}{c} \right) \right)$$
where
$$(( x ))=
\left\{
\begin{aligned}
& x-[x]-\frac{1}{2} & {\rm if\ } x\in \mathbb{R} \setminus \mathbb{Z}, \\
& 0 & {\rm if\ } x\in \mathbb{Z}
\end{aligned} \right.
$$
($[x]$ denotes the integer part of $x$). Also, for $\alpha
=\left( \begin{array}{cc} a & b \\ c & d \end{array}\right) \in SL(2;\mathbb{Z})$,
set $\sigma(\alpha)=\tau_1(\alpha,-1)$, which by a direct computation turns out to be
the signature of the symmetric matrix
$\left( \begin{array}{cc} -2c & a-d \\ a-d & 2b \end{array} \right)$.

\begin{theorem}[Meyer \cite{Mey2}, Satz 4]
\label{thm:Satz4}
For any $\alpha=\left( \begin{array}{cc} a & b \\ c & d \end{array}\right)
\in SL(2;\mathbb{Z})$, we have
$$\phi_1(\alpha)=-\frac{1}{3}\Psi(\alpha)+\sigma(\alpha)\cdot
\frac{1}{2}(1+{\rm sign}(a+d)).$$
In particular, if $a+d \neq 0,1,2$, then $\phi_1(\alpha)=-(1/3)\Psi(\alpha)$.
\end{theorem}

Meyer's proof is based on a certain cocycle identity of $\Psi$, behind which
is the transformation law under $SL(2;\mathbb{Z})$ of the logarithm of the {\it Dedekind $\eta$-function}
$$\eta(\tau)=e^{\pi i \tau/12}\prod_{n=1}^{\infty}(1-e^{2\pi i n \tau}), \quad
\tau \in \{ z\in \mathbb{C}; {\rm Im}(z)>0 \}.$$
Atiyah \cite{Ati2} gave another proof of Theorem \ref{thm:Satz4} of more topological nature.

\subsection{Atiyah's theorem}
\label{sec:Atiyah}

Atiyah \cite{Ati2} showed that the value of $\phi_1$ on hyperbolic
elements coincides with various invariants. Recall that $\alpha \in SL(2;\mathbb{Z})$
is called {\it hyperbolic} if $|{\rm Tr}(\alpha)|>2$.

\begin{theorem}[Atiyah \cite{Ati2}]
\label{thm:car-pool}
For a hyperbolic element $\alpha \in SL(2;\mathbb{Z})$, the following quantities coincide.
\begin{enumerate}
\item $\phi_1(\alpha)$;
\item Hirzebruch's signature defect $\delta(\alpha)$;
\item the transformation low of the logarithm of the Dedekind $\eta$-function under $\alpha$;
\item the logarithmic monodromy of Quillen's determinant line bundle of the mapping torus of $\alpha$;
\item the value $L_{\alpha}(0)$ of the Shimizu L-function;
\item The Atiyah-Patodi-Singer invariant $\eta(\alpha)$ of the mapping torus of $\alpha$;
\item The adiabatic limit $\eta^0(\alpha)$.
\end{enumerate}
\end{theorem}

Since the invariants (6)(7) will appear again in \S \ref{sec:Var}, we give a brief explanation of these invariants here.
The Atiyah-Patodi-Singer invariant \cite{APS}, also called the $\eta$-invariant,
is a spectral invariant of a closed oriented odd dimensional Riemannian manifold
$(M,g)$ and is denoted by $\eta(M,g)$ or $\eta(M)$ shortly. Further, let $E$ and $B$ be closed oriented $C^{\infty}$-manifolds and
$\pi\colon E\to B$ a oriented $M$-bundle with the dimension of $E$ is divisible by 4.
Once a metric $g^{E/B}$ on the relative tangent bundle $T(E/B)$, a metric $g^B$ on $B$,
and a connection $\nabla$ on $TE$ are given, the metric on $E$ is given by $g^E:=g^{E/B}\oplus \pi^*g^E$
according to the decomposition $TE=T(E/B)\oplus \pi^*TB$ induced from $\nabla$.
Then the one parameter family of metrics on $E$ is defined by
$g^E_{\varepsilon}:=g^{E/B}\oplus \varepsilon^{-1}\pi^*g^B$, $\varepsilon \in \mathbb{R}_{>0}$.
By Bismut-Cheeger \cite{BC}, it is shown that the limit $\lim_{\varepsilon \to 0} \eta(E,g^E_{\varepsilon})$
exists. The limit is called the {\it adiabatic limit} of the $\eta$-invariants and is denoted by
$\eta^0(E)$. In Theorem \ref{thm:car-pool}, a suitable metric is chosen for the mapping torus of $\alpha$.

In fact, Atiyah also showed the following result, giving an analytic expression of the value of
$\phi_1$ on any element of $SL(2;\mathbb{Z})$.

\begin{theorem}[Atiyah \cite{Ati2}]
\label{thm:phi=eta}
For $\alpha\in SL(2;\mathbb{Z})$, we have
$\phi_1(\alpha)=\eta^0(\alpha)$.
\end{theorem}

A generalization of this result to $\phi_2$ will be dealt in \S \ref{sec:theta}.

\section{Local signatures}

Consider a closed oriented 4-manifold $M$ admitting a fibration
$f\colon M\to B$ onto a closed oriented surface $B$. Under some conditions,
the signature of $M$ happens to {\it localize} to finitely many
singular fibers of $f$. This phenomenon is called the {\it localization
of the signature}, and has been studied from several point of view.
In this section we review some of these treatments, and recall an approach
using Meyer functions.  

\subsection{Local signatures and Horikawa index}
\label{subsec:Horikawa}

Let $E$ and $B$ be compact oriented $C^{\infty}$-manifolds of dimension 4 and 2
respectively, $f\colon E\to B$ a proper surjective $C^{\infty}$-map having
the structure of $\Sigma_g$-bundle outside of finitely many points
$\{b_i \}_{i\in I} \subset {\rm Int}(B)$.
We call such a triple $(E,f,B)$ a {\it fibered 4-manifold} (of genus $g$).
For $b\in B$, we denote by $\mathcal{F}_b$ the fiber germ of $f$ around $b$.
If $b\in B\setminus \{ b_i\}_{i\in I}$, $\mathcal{F}_b$ is called a {\it general fiber}.
If $b=b_i$ for some $i\in I$, $\mathcal{F}_b$ is called a {\it singular fiber}.

Typical examples of fibered 4-manifolds are elliptic surfaces and Lefschetz fibrations.
When we work with holomorphic category, then $E$ is a complex surface,
$B$ is a Riemann surface, and $f$ is a holomorphic map. In this case if we say,
for example, that $f\colon E\to B$ is a hyperelliptic fibration, then general fibers
are hyperelliptic Riemann surfaces.

Among the topological invariants of such $E$, the topological Euler number $\chi(E)$
is easy to compute. For simplicity we assume that $E$ and $B$ are closed,
and let $g(B)$ be the genus of $B$.
Let $\Delta_i \subset B$ be a small closed disk with center $b_i$
and we denote $E_i=f^{-1}(\Delta_i)$ and $E_0=f^{-1}(B\setminus \bigcup_i {\rm Int}(\Delta_i))$.
Since the topological Euler number is multiplicative in fiber bundles,
we have $\chi(E_0)=(2-2g)(2-2g(B)-|I|)$. Moreover, since $f$ is proper
we have $\chi(E_i)=\chi(f^{-1}(b_i))$. Thus
$$\chi(E)=(2-2g)(2-2g(B))+\sum_{b\in B}\varepsilon(\mathcal{F}_b),$$
where the number $\varepsilon(\mathcal{F}_b):=\chi(f^{-1}(b_i))-(2-2g)$ is called
the {\it topological Euler contribution}. In short, we can compute $\chi(E)$
by the contributions $\varepsilon(\mathcal{F}_b)$.

On the other hand, the signature of $E$ is not so easy to compute and
in general one cannot compute it from the data of singular fiber germs.
Nevertheless, under some conditions on the general fibers,
it happens that we can assign
a rational number $\sigma(\mathcal{F}_b)$ to each fiber $\mathcal{F}_b$
satisfying the following two conditions:
\begin{enumerate}
\item if $\mathcal{F}_b$ is a general fiber, then $\sigma(\mathcal{F}_b)=0$.

\item if $E$ is closed, then ${\rm Sign}(E)=\sum_{b\in B} \sigma(\mathcal{F}_b)$.
\end{enumerate}
The assignment $\sigma$ is called a {\it local signature}, and
when such phenomena happens, we say that the signature of $E$ is localized.

The first example of a local signature is the one for fibered 4-manifolds of genus 1
due to Y. Matsumoto \cite{Mat1}. He called such assignment a fractional signature.
Later he also gave a local signature for Lefschetz fibrations of genus 2 \cite{Mat2}.
In both the examples, he used the Meyer functions $\phi_1$ and $\phi_2$
to construct a local signature. See the next subsection for details.

In algebro-geometric setting, local signatures are closely related to an invariant
of fiber germs which originates in the work of Horikawa \cite{Hori1} \cite{Hori2}.
He studied global family of curves of genus 2 $f\colon E \to B$ and defined
an invariant $\mathcal{H}(\mathcal{F}_b)\ge 0$ to each fiber germ, and showed the equality
\begin{equation}
\label{H-index}
K_E^2=2\chi(\mathcal{O}_E)-6+6g(B)+\sum_{b\in B} \mathcal{H}(\mathcal{F}_b).
\end{equation}
Here $g(B)$ is the genus of $B$,
$K_E^2$ is the self intersection number of the canonical bundle of $E$,
and $\chi(\mathcal{O}_E)$ is the Euler characteristic number of the structure sheaf of $E$.
In the geography of complex surfaces of general type, one often studies complex surfaces
with the pair of specified numerical invariants $(K_E^2,\chi(\mathcal{O}_E))$.
Note that by the Hirzebruch signature formula ${\rm Sign}(E)=(1/3)(K_E^2-2\chi(E))$
and the Noether formula $\chi(\mathcal{O}_E)=(1/12)(K_E^2+\chi(E))$,
to fix $(K_E^2,\chi(\mathcal{O}_E))$ is equivalent to fix $({\rm Sign}(E),\chi(E))$. The inequality
$K_E^2 \ge 2\chi(\mathcal{O}_E)-6$ is called the {\it Noether inequality}, a
lower bound for the numerical invariants of complex surfaces of general type.
Thus $\mathcal{H}(\mathcal{F}_b)$ is regarded as a local contribution of each fiber germ
to the distance from the geographical lower bound for $(K_E^2,\chi(\mathcal{O}_E))$.
The invariant $\mathcal{H}(\mathcal{F}_b)$ is called the {\it Horikawa index}.

There are several situations in which the Horikawa index exists.
M. Reid \cite{Reid} defined it for fiber germs of non-hyperelliptic fibrations of genus 3.
This is generalized by Konno \cite{Konno2} to Clifford general fibrations of odd genus.

Arakawa and Ashikaga \cite{AA} introduced the Horikawa index for hyperelliptic fibrations,
which is regarded as a direct generalization of the work of Horikawa. Let $f\colon E\to B$
be a hyperelliptic fibration of genus $g$ with $B$ closed. They introduced an invariant 
$\mathcal{H}(\mathcal{F}_b)\ge 0$ for each fiber germ satisfying
\begin{equation}
\label{H-ind-hyp}
K_{E/B}^2=\frac{4(g-1)}{g}\chi_f+\sum_{b\in B}\mathcal{H}(\mathcal{F}_b ),
\end{equation}
where $K_{E/B}^2=K_S^2-8(g-1)(g(B)-1)$ and $\chi_f=\chi(\mathcal{O}_E)-(g-1)(g(B)-1)$.
Moreover, they defined a local signature for hyperelliptic fibrations of genus $g$ by
\begin{equation}
\label{eq:alg-loc}
\sigma_g^{{\rm alg}}(\mathcal{F}_b):=\frac{1}{2g+1}(g\mathcal{H}(\mathcal{F}_b)
-(g+1)\varepsilon(\mathcal{F}_b)).
\end{equation}
Here $\varepsilon(\mathcal{F}_b)$ is the topological Euler contribution as above.
That $\sigma_g^{{\rm alg}}$ is a local signature follows from
(\ref{H-ind-hyp}). More generally, if we find a Horikawa index in a class of fibrations
(say non-hyperelliptic fibrations of genus 3), then a formula of type (\ref{eq:alg-loc})
gives a local signature for such fibrations.

For more detail about local signatures, we refer to the survey articles Ashikaga-Endo \cite{AE} and
Ashikaga-Konno \cite{AK}. We also refer to recent works by Ashikaga-Yoshikawa \cite{AY}
and Sato \cite{Sato}.

\subsection{Matsumoto's formula}
\label{sec:Mat}

For a while we assume $g$ is $1$ or $2$. Let $(E,f,B)$ be a fibered 4-manifold of genus $g$.
For each $b\in B$, take a small closed disk neighborhood $\Delta \subset B$ of $b$ and consider the restriction of $f$ to
$\Delta\setminus \{b \}$. Let $x_b\in \mathcal{M}_g$ be the monodromy of this $\Sigma_g$-bundle
along the boundary $\partial \Delta$ with the counter-clockwise orientation, and set
\begin{equation}
\label{eq:mat}
\sigma_g(\mathcal{F}_b):=\phi_g(x_b)+{\rm Sign}(f^{-1}(\Delta))\in \mathbb{Q}.
\end{equation}
Here $\phi_g$ is the Meyer function of genus $g$.
Note that although $x_b$ is only defined up to conjugacy, $\phi_g(x_b)$ is well defined by Lemma \ref{lem:phi} (3).

\begin{proposition}[Y. Matusmoto \cite{Mat1} \cite{Mat2}]
\label{prop:Mat-loc}
Let $g=1$ or $2$. The assignment
$\sigma_g(\mathcal{F}_b)$ is a local signature for fibered 4-manifolds of genus $g$.
\end{proposition}

\begin{proof}
The property (1) is clear since $x_b$ is trivial if $\mathcal{F}_b$ is non-singular.
To prove (2), for each $i$ let $\Delta_i$ be a small closed disk neighborhood of $b_i$.
By Proposition \ref{prop:s(E)}, we have
\begin{align*}
{\rm Sign}(E) &={\rm Sign}(f^{-1}(B_0))+\sum_{i\in I} {\rm Sign}(f^{-1}(\Delta_i)) \\
&=\sum_{i\in I}\phi_g(x_{b_i})+\sum_{i\in I}{\rm Sign}(f^{-1}(\Delta_i)=\sum_{i\in I}\sigma(\mathcal{F}_{b_i}).
\end{align*}
\end{proof}

Matsumoto \cite{Mat1} \cite{Mat2} also gave some computations of his local signatures.
Using the Meyer function on the hyperelliptic mapping class group and applying
the formula (\ref{eq:mat}), Endo \cite{End} introduced a local signature for
hyperelliptic fibrations (see \S \ref{HMP}). By Terasoma, it was shown that Endo's local signature and
Arakawa-Ashikaga's local signature (\ref{eq:alg-loc}) coincide.  See \cite{End} Appendix.

The formula (\ref{eq:mat}) implies that the local signature is only determined by
topological data. But as Konno \cite{Konno1} observed, there exists a topologically
non-singular fiber germ of non-hyperelliptic fibrations of genus 3 which has a
non-zero Horikawa index. In fact, in the central fiber $f^{-1}(b)$ of Konno's example
is a non-singular hyperelliptic curve of genus 3. From the view point of local signatures,
this fiber germ should be thought as a singular fiber. A modification of the formula
(\ref{eq:mat}) for such situations will be explained in \S \ref{MPV}.

\section{Variations}
\label{sec:Var}

In this section we review higher genera analogues
and higher dimensional analogues of Meyer's $\phi_1$ or $\phi_2$.
First note that by Theorem \ref{thm:Satz2},
Meyer functions does not exist on $\mathcal{M}_g$ for $g>2$.
But the signature cocycle happens to be a coboundary
when it is pulled back to some group, for example, a subgroup of $\mathcal{M}_g$.
The examples in \S \ref{HMP} and \S \ref{MPV} are those of this kind. The example
in \S \ref{sec:theta} is in a situation of Remark \ref{var-sig}, and can be regard
as a generalization of Theorem \ref{thm:phi=eta}.

\subsection{Hyperelliptic mapping class group}
\label{HMP}

Let $\iota \in \mathcal{M}_g$ be a {\it hyperelliptic involution}, i.e., (the class of)
an involution of $\Sigma_g$ acting on $H_1(\Sigma_g;\mathbb{Z})$ as $-{\rm id}$. The
{\it hyperelliptic mapping class group} $\mathcal{H}_g$ is the centralizer of $\iota$:
$$\mathcal{H}_g:=\{ f\in \mathcal{M}_g ; f\iota=\iota f \}.$$

Let $\tau_g^H \in Z^2(\mathcal{H}_g;\mathbb{Z})$ be the restriction of $\tau_g$
to the subgroup $\mathcal{H}_g \subset \mathcal{M}_g$. Using a finite presentation
of $\mathcal{H}_g$ by Birman-Hilden \cite{BH} and Proposition \ref{prop:comb.cri},
Endo \cite{End} proved the following theorem.

\begin{theorem}[Endo \cite{End}]
The order of $[\tau_g^H] \in H^2(\mathcal{H}_g;\mathbb{Z})$ is $2g+1$. Furthermore,
there uniquely exists a function $\phi_g^H\colon \mathcal{H}_g \to (1/2g+1)\mathbb{Z}$
such that $\delta \phi_g^H=\tau_g^H$.
\end{theorem}

The 1-cochain $\phi_g^H$ is called the {\it Meyer function for the hyperelliptic
mapping class group of genus $g$}.

\begin{remark}
The existence and uniqueness of $\phi_g^H$ also follow from the
$\mathbb{Q}$-acyclicity of $\mathcal{H}_g$ which is independently proved by
Cohen \cite{Coh} and Kawazumi \cite{Kaw1}.
\end{remark}

Remark that $\mathcal{H}_g=\mathcal{M}_g$ if $g=1$ or $2$. Thus the series $\phi_g^H$, $g\ge 3$
could be a higher genus analogue of Meyer's $\phi_1$ and $\phi_2$.
The values of $\phi_g^H$ on Dehn twists are given as follows (\cite{End} \cite{Morifuji}).
Let $C$ be an $\iota$-invariant simple closed curve on $\Sigma_g$.
We denote by $t_C$ the right handed Dehn twist along $C$, which is an element of $\mathcal{H}_g$.
If $C$ is non-separating, then $\phi_g^H(t_C)=(g+1)/2g+1$; if $C$ is separating
and separates $\Sigma_g$ into surfaces of genus $h$ and $g-h$, then
$\phi_g^H(t_C)=-4h(g-h)/2g+1$.

A fibered 4-manifold $(E,f,B)$ is called {\it hyperelliptic} if the monodromy
of the $\Sigma_g$-bundle over $B\setminus \{ b_i \}_i$ can take value in $\mathcal{H}_g$
by a suitable identification of a reference fiber with $\Sigma_g$.
Replacing $\phi_g$ with $\phi_g^H$ in (\ref{eq:mat}), Endo \cite{End}
introduced a local signature for hyperelliptic fibered 4-manifold.

Morifuji \cite{Morifuji} studied geometrical aspects of $\phi_g^H$.
He showed if $f\in \mathcal{H}_g$ is of finite order, then $\phi_g^H(f)$ equals
$\eta(f)$, the $\eta$-invariant (see \S \ref{sec:Atiyah}) of the mapping torus
$\Sigma_g \times [0,1]/(x,0)\sim (f(x),1)$.
Further, he showed that $\phi_g^H(f)=d_0(f)$ if
$f$ belongs to the hyperelliptic Torelli group, where $d_0$ is the so-called core
of the Casson invariant introduced by Morita \cite{Morita2} \cite{Morita3}.

\subsection{Family of smooth theta divisors}
\label{sec:theta}

Iida \cite{Iida} gave a higher dimensional analogue of Meyer's $\phi_2$,
which he called the {\it Meyer function for smooth theta divisors}.

Let $\mathfrak{S}_g:=\{ \tau \in M(g;\mathbb{C}); \ ^t \tau=\tau, {\rm Im}(\tau)>0 \}$
be {\it the Siegel upper half space} of degree $g$ and $f\colon \mathbb{A}_g\to \mathfrak{S}_g$
the universal family of principally polarized Abelian varieties. The fiber of $f$ at
$\tau\in \mathfrak{S}_g$ is the complex torus $A_{\tau}=\mathbb{C}^g/\Lambda_g$,
where $\Lambda_g$ is the lattice spanned by the column vectors of
the $g\times 2g$ matrix $(I_g \ \tau )$. We denote ${\bf e}(t)=\exp(2\pi \sqrt{-1}t)$.
The Riemann theta function
$$\theta(z,\tau):=\sum_{n\in \mathbb{Z}^g}
{\bf e}\left( \frac{1}{2}n\tau \ ^tn+n \ ^tz \right), \quad z\in \mathbb{C}^g,$$
defines a holomorphic section of a certain holomorphic vector bundle on $A_{\tau}$
and its zero locus is called the {\it theta divisor}. Set
$$\Theta:=\{ (z,\tau); \tau\in \mathfrak{S}_g, z\in A_{\tau}, \theta(z,\tau)=0 \}$$
and let $p\colon \Theta \to \mathfrak{S}_g$ be the natural projection. This is
the universal family of theta divisors. We denote by $\Theta_{\tau}$ the fiber
of $p$ at $\tau$. The group $Sp(2g;\mathbb{Z})$,
which for simplicity we denote here by $\Gamma_g$, naturally acts on $\mathfrak{S}_g$.
Iida introduced a $\Gamma_g$-action on $\Theta$ so that $p$ is $\Gamma_g$-equivariant.

The Zariski closed set
$\mathcal{N}_g:=\{ \tau\in \mathfrak{S}_g; {\rm Sing}(\Theta_{\tau}) \neq \emptyset \}$
is called the {\it Andreotti-Mayer} locus. The group $\Gamma_g$ acts on
the complement $\mathfrak{S}_g^{\circ}:=\mathfrak{S}_g \setminus \mathcal{N}_g$
properly discontinuously. Let $\mathcal{S}_g$ be the orbifold fundamental group
of the quotient orbifold $\Gamma_g \backslash \mathfrak{S}_g^{\circ}$.
In other words, $\mathcal{S}_g$ is the fundamental group of
the Borel construction
$(\mathfrak{S}_g^{\circ})_{\Gamma_g}:=E\Gamma_g \times_{\Gamma_g} \mathfrak{S}_g^{\circ}$,
where $E\Gamma_g$ is the total space of the classifying space of $\Gamma_g$.
The group $\mathcal{S}_g$ fits into an exact sequence
\begin{equation}
\label{S_g}
1\to \pi_1(\mathfrak{S}_g^{\circ})\to \mathcal{S}_g \to \Gamma_g \to 1.
\end{equation}
If $g=1$, $\mathcal{N}_g =\emptyset$ and $\Gamma_1 \backslash \mathfrak{S}_1^{\circ}$
is the moduli space of curves of genus 1, hence $\mathcal{S}_1=\mathcal{M}_1$.
By the Torelli theorem, $\Gamma_2 \backslash \mathfrak{S}_2^{\circ}$
is the moduli space of curves of genus 2 and $\mathcal{S}_2=\mathcal{M}_g$.

The projection $p$ induces a fiber bundle over $(\mathfrak{S}_g^{\circ})_{\Gamma_g}$.
The fiber is diffeomorphic to a smooth theta divisor.
By the construction given in Remark \ref{var-sig}, we get the signature cocycle
$c_g\colon \mathcal{S}_g \times \mathcal{S}_g \to \mathbb{Z}$.
If $g$ is odd, $c_g \equiv 0$ since the real dimension of a smooth theta divisor is $2g-2$.
When $g=2$, $c_2$ is the pull back of $\tau_2^{\rm sp}$ by (\ref{S_g}).
But if $g\ge 3$, this is not the case.

Using adiabatic limits of $\eta$-invariants and a certain automorphic form,
Iida constructed a 1-cochain of $\mathcal{S}_g$ which cobounds $c_g$.
Suppose $g$ is even.
An element $\sigma\in \mathcal{S}_g$ can be written as $\sigma=(\alpha,\gamma)$,
where $\alpha \colon [0,1]\to \mathfrak{S}_g^{\circ}$ is a continuous map
with $\alpha(0)$ a specified basepoint of $\mathfrak{S}_g^{\circ}$ and
$\gamma \in \Gamma_g$ such that $\alpha(1)=\gamma\cdot \alpha(0)$.
Consider the mapping torus $M_{\sigma}:=[0,1]\times_{\alpha}\Theta/(0,x)\sim (1,\gamma x)$
and the projection $\pi\colon M_{\sigma}\to S^1=[0,1]/0\sim 1$.
He introduced a metric of the relative tangent bundle $T(M_{\sigma}/S^1)$
and a connection on $M_{\sigma}$. Then the adiabatic limit $\eta^0(M_{\sigma})$ is defined
(see \S \ref{sec:Atiyah}). Set
$$\Phi_g(\sigma):=\eta^0(M_{\sigma})+
\frac{(-1)^{\frac{g}{2}}2^{g+3}(2^{g+2}-1)}{(g+3)!}B_{\frac{g}{2}+1}
\int_{S^1}\alpha^* d^c \log ||\Delta_g(\tau)||.$$
Here $\Delta_g(\tau)$ is a Siegel cusp form of weight $(g+3)g!/2$ with zero divisor $\mathcal{N}_g$
and $B_k$ is the $k$-th Bernoulli number.

\begin{theorem}[Iida \cite{Iida}]
The 1-cochain $\Phi_g$ cobounds $c_g$, i.e.,
$$c_g(\sigma_1,\sigma_2)=\Phi_g(\sigma_1)-\Phi_g(\sigma_1\sigma_2)+\Phi_g(\sigma_2), \quad
\sigma_1,\sigma_2\in \mathcal{S}_g.$$
\end{theorem}

It should be remarked that the uniqueness of $\Phi_g$ does not hold.
In fact, Iida proved that $H^1(\mathcal{S}_g;\mathbb{Z})=\mathbb{Z}$ for $g\ge 4$
(\cite{Iida} Theorem 13). The 1-cochain $c_g$ actually takes values in $\mathbb{Q}$ (\cite{Iida} Theorem 15 ).
As a special case, Iida obtained an analytic expression of the Meyer function of genus 2.

\begin{corollary}[Iida \cite{Iida}] For $\sigma=(\alpha,\gamma)\in \mathcal{S}_2=\mathcal{M}_2$,
we have
$$\phi_2(\sigma)=\eta^0(M_{\sigma})-\frac{2}{15}\int_{S^1}
\alpha^* d^c \log || \chi_2(\tau)||^2.$$
Here $\chi_2(\tau)$ is a Siegel modular form of weight $5$ called the {\it Igusa modular form}.
\end{corollary}

\subsection{The Meyer functions for projective varieties}
\label{MPV}

We mention an approach by Kuno \cite{Kuno1} \cite{Kuno2} to
extend Matsumoto's formula (\ref{eq:mat}) for generic non-hyperelliptic fibrations
of small genera.

Let $X\subsetneq \mathbb{P}_N$ be a smooth projective variety of dimension $n\ge 2$,
embedded in a complex projective space of dimension $N$. The intersection of $X$
and a generic plane in $\mathbb{P}_N$ of codimension $n-1$ is non-singular of dimension 1.
Set $k:=N-n+1$ and let
$G_k(\mathbb{P}_N)$ be the Grassmann manifold of $k$-planes of $\mathbb{P}_N$.
The set
$$D_X:=\{ W\in G_k(\mathbb{P}_N); W {\rm \ meets\ } X {\rm \ not\ transversally\ } \}$$
is called the $k$-th {\it associated subvariety} of $X$ \cite{GKZ}. Over the complement
$U^X:=G_k(\mathbb{P}_N) \setminus D_X$, there is a family of compact Riemann surfaces
$p_X\colon \mathcal{C}^X\to U^X$ whose fiber at $W\in U^X$ is $X\cap W$.
Let $g$ be the genus of the fibers and let
$\rho_X\colon \pi_1(U^X)\to \mathcal{M}_g$ be the monodromy of this family.

\begin{theorem}[Kuno \cite{Kuno2}]
There exists a unique $\mathbb{Q}$-valued 1-cochain
$\phi_X\colon \pi_1(U^X)\to \mathbb{Q}$ whose coboundary equals
the pull-back $\rho_X^*\tau_g$.
\end{theorem}

The 1-cochain $\phi_X$ is called the {\it Meyer function associated to $X\subset \mathbb{P}_N$}.
The fundamental group $\pi_1(U^X)$ is {\it normally} generated by a single
element called a {\it lasso}, which is represented by a loop ``going once around $D_X$".
By $\rho_X$, a lasso is mapped to a Dehn twist.
By a certain extension of theory of Lefschetz pencils,
the value of $\phi_X$ on a lasso is given in terms of invariants of $X$.
Under a mild condition on $X$,
it follows that $\phi_X$ is an unbounded function. As a consequence, we can show that
the group $\pi_1(U^X)$ is non-amenable for such $X$.

As an application, we can define a local signature for generic non-hyperelliptic
fibrations of small genera. Let us illustrate this by an example. Let $(E,f,B)$ be a fibered
4-manifold of genus 3, such that the restriction of $f$ to $B\setminus \{ b_i \}_{i\in I}$
is a continuous family of Riemann surfaces with fiber non-hyperelliptic. We call such
$(E,f,B)$ a {\it non-hyperelliptic fibration of genus 3}. Note that we assume a fiberwise
complex structure on the general fibers, but do not assume a global complex structure.
The idea is to construct a certain universal family and
to lift the monodromy to the fundamental group of the base space of it.

Hereafter let $X$ be the image of the Veronese embedding $v_4\colon \mathbb{P}_2\to \mathbb{P}_{14}$
of degree 4. A generic hyperplane section of $\mathbb{P}_{14}$ corresponds to a smooth plane curve of
degree 4 in $\mathbb{P}_2$, which is non-hyperelliptic of genus 3. The group $\mathcal{G}=PGL(3)$
naturally acts on $\mathbb{P}_{14}$ preserving $D_X$. This induces $\mathcal{G}$-actions on
$\mathcal{C}_X$ and $U^X$, making $p_X\colon \mathcal{C}^X\to U^X$ a $\mathcal{G}$-equivariant map.
Therefore we have a continuous family of non-hyperelliptic Riemann surfaces of genus 3
over the Borel construction $U^X_{\mathcal{G}}:=E\mathcal{G} \times_{\mathcal{G}} U^X$,
which we denote by $p_u\colon \mathcal{C}^X_{\mathcal{G}}\to U^X_{\mathcal{G}}$.
This family has a certain universal property: if $p\colon E\to B$ is a continuous family of
non-hyperelliptic Riemann surfaces of genus 3, then there exist a continuous map
$g\colon B\to U^X_{\mathcal{G}}$ such that the fiber product $\mathcal{C}^X \times_g B$ and
the original family are {\it isotopic}. Moreover, such $g$ is unique up to homotopy.
The fundamental group $\pi_1(U^X_{\mathcal{G}})$ fits into an exact sequence
$$\pi_1(PGL(3))\cong \mathbb{Z}/3\mathbb{Z} \to \pi_1(U^X) \to \pi_1(U^X_{\mathcal{G}}) \to 1.$$
From this and the existence of $\phi_X$ on $\pi_1(U^X)$, we can deduce that there exists
a unique $\mathbb{Q}$-valued 1-cochain $\phi_3^{NH}\colon \pi_1(U^X_{\mathcal{G}})\to \mathcal{Q}$
which cobounds the pull-back of $\tau_3$ by the monodromy $\rho_u\colon \pi_1(U^X_{\mathcal{G}}) \to \mathcal{M}_3$.

Now, let $\mathcal{F}_b$ be a fiber germ of non-hyperelliptic fibration of genus 3.
Take a small closed disk $\Delta$ with center $b$,
so that there is no singular fiber on $\Delta \setminus \{ b\}$.
By the universality of $p_u$, there is a continuous map
$g_{\mathcal{F}_b}\colon \Delta \setminus \{ b\} \to U^X_{\mathcal{G}}$.
Set $x_{\mathcal{F}_b}:=(g_{\mathcal{F}_b})_*(\partial \Delta)\in \pi_1(U^X_{\mathcal{G}})$,
where we give $\partial \Delta$ the counterclockwise orientation.
Note that $x_{\mathcal{F}_b}$ is uniquely determined up to conjugacy.
Set
$$\sigma_3^{NH}(\mathcal{F}_b):=\phi_3^{NH}(x_{\mathcal{F}_b})+{\rm Sign}(f^{-1}(\Delta)).$$
By applying the proof of Proposition \ref{prop:Mat-loc}, we have the following.

\begin{theorem}[\cite{Kuno1}]
The assignment $\sigma_3^{NH}$ is a local signature for
non-hyperelliptic fibrations of genus 3.
\end{theorem}

The formulation of $\sigma_3^{NH}$ gives a topological interpretation of Konno's
example in \S \ref{sec:Mat}. While the monodromy around $b$ is trivial, its
lift $x_{\mathcal{F}_b}\in \pi_1(U^X_{\mathcal{G}})$ is non-trivial and contributes
to $\sigma_3^{NH}$. Similar constructions are possible for generic non-hyperelliptic
fibrations of genus 4 and 5. For details, see \cite{Kuno2}.

\vspace{0.5cm}
\noindent \textbf{Acknowledgments.} The author would like
to thank Tadashi Ashikaga for helpful comments on an earlier
draft of this paper.

\end{document}